\documentclass[11pt]{article}
\usepackage{amsmath}
\usepackage{amssymb}
\usepackage{amsthm}
\usepackage{mathabx}
\usepackage[usenames]{color}
\usepackage{amscd}
\usepackage{dsfont}
\usepackage{indentfirst}

\usepackage[colorlinks=true,linkcolor=blue,filecolor=red,
citecolor=webgreen]{hyperref}
\definecolor{webgreen}{rgb}{0,.5,0}

\def\C{{\mathds{C}}}

\def\R{{\mathbb{R}}}
\def\N{{\mathds{N}}}

\def\1{{\bf 1}}

\def\id{\operatorname{id}}
\def\lcm{\operatorname{lcm}}

\def\DOT{\text{\rm\Huge{.}}}

\newtheorem{theorem}{Theorem}
\newtheorem{thm}[theorem]{Theorem}
\newtheorem{lemma}[theorem]{Lemma}
\newtheorem{cor}[theorem]{Corollary}
\newtheorem{Con}[theorem]{Conjecture}

\newtheorem{prop}[theorem]{Proposition}
\newtheorem{remark}[theorem]{Remark}
\newtheorem{question}[theorem]{Question}

\begin{document}

\title{{\bf Unitary cyclotomic polynomials}}
\author{Pieter Moree and L\'aszl\'o T\'oth}
\maketitle

\begin{abstract}
The notion of block divisibility naturally leads one to introduce
unitary cyclotomic polynomials.
We formulate some basic properties of unitary cyclotomic polynomials and study
how they are connected with cyclotomic,
inclusion-exclusion and Kronecker polynomials. Further, we derive
some related arithmetic function identities involving the unitary analog of the Dirichlet
convolution.
\end{abstract}

\tableofcontents

\section{Introduction}

\subsection{Unitary divisors}

A divisor $d$ of $n$ ($d,n\in \N$) is called a \emph{unitary divisor} (or block divisor) if $(d,n/d)=1$, notation $d\mid\mid n$
(note that this is in agreement with the standard notation $p^a \mid\mid n$ used for prime powers $p^a$). If the prime power factorization of
$n$ is $n=p_1^{a_1}\cdots p_s^{a_s}$, then the
set of its unitary divisors consists of the
integers $d=p_1^{b_1}\cdots p_s^{b_s}$, where $b_i=0$ or $b_i=a_i$ for any $1\le i\le s$.

\par The study of \emph{arithmetic functions} defined by unitary divisors goes back to Vaidyanathaswamy \cite{Vai1931} and Cohen \cite{Coh1960}.
For example, the analogs of the sum-of-divisors function $\sigma$ and Euler's totient function $\varphi$ are
$\sigma^*(n)=\sum_{d\mid\mid n} d$, respectively $\varphi^*(n)=\# \{j: 1\le j \le n, (j,n)_*=1\}$, where
\begin{equation*}
(j,n)_*=\max \{d: d\mid j,\, d\mid\mid n\}.
\end{equation*}

Several properties of the unitary functions $\sigma^*$ and $\varphi^*$ run parallel to
those of $\sigma$ and $\varphi$, respectively. For example, both functions $\sigma^*$ and $\varphi^*$ are multiplicative, and
$\sigma^*(p^a)= p^a+1$, $\varphi^*(p^a)=p^a-1$ for prime
powers $p^a$ ($a \ge 1$). The \emph{unitary convolution} of the functions $f$ and $g$ is defined by
\begin{equation*}
(f\times g)(n)=\sum_{d\mid\mid n} f(d)g(n/d) \quad (n\in \N).
\end{equation*}

The set of arithmetic functions $f$ such that $f(1)\ne 0$ forms a commutative group under the unitary convolution and the set of multiplicative
functions is a subgroup. The identity is the function $\epsilon$, given by $\epsilon(1)=1$, $\epsilon(n)=0$ ($n>1$), similar to the case of
Dirichlet convolution. The inverse of the constant $1$ function under the unitary convolution is $\mu^*(n)=(-1)^{\omega(n)}$, where
$\omega(n)$ denotes the number of distinct prime
factors of $n$. That is,
\begin{equation} \label{unit_Mobius}
\sum_{d\mid \mid n} \mu^*\left(\frac{n}{d}\right)=\sum_{d\mid \mid n} \mu^*(d)= \epsilon(n) \quad (n\in \N).
\end{equation}
See, e.g., the books by Apostol \cite{Apostolbook}, McCarthy \cite{McC1986} and Sivaramakrishnan \cite{Siv1989}.

\subsection{Unitary Ramanujan sums}

The \emph{unitary Ramanujan sums} $c^*_n(k)$ were defined by Cohen \cite{Coh1960} as follows:
\begin{equation*}
c^*_n(k) = \sum_{\substack{1\le j\le n\\ (j,n)_*=1}} \zeta_n^{jk} \quad (k,n\in \N),
\end{equation*}
where $\zeta_n:=e^{2\pi i/n}.$ (The classical Ramanujan sums are defined similarly, but with $(j,n)_*=1$
replaced by $(j,n)=1.$)

The identities
\begin{equation}\label{unit_c_1}
c^*_n(k) = \sum_{d\mid\mid (k,n)_*} d\mu^*(n/d) \quad (n,k\in \N),
\end{equation}
\begin{equation} \label{unit_c_2}
\sum_{d\mid\mid n} c^*_d(k) = \varrho_n(k)= \begin{cases} n & \text{ if $n\mid k$;} \\ 0 & \text{ otherwise,}
\end{cases}
\end{equation}
can be compared to the corresponding ones concerning the classical Ramanujan sums $c_n(k)$. Note that
$c^*_n(n)=\varphi^*(n)$, $c^*_n(1)=\mu^*(n)$ ($n\in \N$).

\subsection{Unitary cyclotomic polynomials}

The \emph{cyclotomic polynomials} $\Phi_n(x)$ are defined by
\begin{equation}
\label{definitie}
\Phi_n(x)=\prod_{\substack{j=1\\ (j,n)=1}}^n \left(x- \zeta_n^j \right).
\end{equation}

They arise as irreducible factors (see Weintraub \cite{Wein})  on
factorizing $x^n-1$ over the rationals:
\begin{equation}
    \label{factor}
    x^n-1=\prod_{d\mid n}\Phi_d(x).
\end{equation}

By M\"obius inversion it follows from \eqref{factor} that
\begin{equation} \label{Phi}
\Phi_n(x)=\prod_{d\mid n} \left(x^{n/d}-1\right)^{\mu(d)}=\prod_{d\mid n} \left(x^{d}-1\right)^{\mu(n/d)},
\end{equation}
where $\mu$ denotes the M\"obius function.

The \emph{unitary cyclotomic polynomial} $\Phi^*_n(x)$ is defined by
\begin{equation} \label{def_unit_cyclotomic}
\Phi^*_n(x)=\prod_{\substack{j=1\\ (j,n)_*=1}}^n \left(x- \zeta_n^j \right),
\end{equation}
see \cite[Ch.\ X]{Siv1989}. It is monic, has integer coefficients and is of degree $\varphi^*(n)$. Furthermore, for any
natural number $n$ we have
\begin{equation} \label{prod_Phi_star}
x^n-1 = \prod_{d\mid\mid n} \Phi^*_d(x)
\end{equation}
and
\begin{equation} \label{Phi_star}
\Phi^*_n(x)=\prod_{d\mid\mid n} \left(x^{n/d}-1\right)^{\mu^*(d)}=\prod_{d\mid\mid n} \left(x^{d}-1\right)^{\mu^*(n/d)}.
\end{equation}

See Section \ref{basicunitary} for short direct proofs of these properties and further
basic properties of unitary cyclotomic
polynomials.

If $n$ is squarefree, then the unitary divisors of $n$ coincide with the divisors of $n$ and hence comparing
\eqref{Phi} with \eqref{Phi_star} yields
$\Phi^*_n(x)=\Phi_n(x)$. In this case $\Phi^*_n(x)$ is irreducible
over the rationals.
However, a quick check shows, that for certain non-squarefree values of $n$, the polynomial $\Phi^*_n(x)$ is reducible over the rationals.
For example,
$\Phi^*_{12}(x) = \Phi_6(x)\Phi_{12}(x)$ and $\Phi^*_{40}(x) = \Phi_{10}(x)\Phi_{20}(x)\Phi_{40}(x)$.
Indeed, we will show that $\Phi_n^*$ is reducible
for \emph{every} non-squarefree integer $n$.
This is a corollary of the fact
that each polynomial $\Phi^*_n(x)$ can be written as
the product of the cyclotomic
polynomials $\Phi_d(x)$, where $d$ runs over the divisors of $n$ such that $\kappa(d)=\kappa(n)$,
with $\kappa(n)$ the \emph{squarefree kernel} of $n$
(Theorem \ref{Th_pol}).
In fact, this is a consequence of a more general result (Theorem \ref{Th_general}) involving unitary divisors.

One can introduce the \emph{bi-unitary cyclotomic polynomials} $\Phi^{**}_n(x)$ defined by
\begin{equation*}
\Phi^{**}_n(x)=\prod_{\substack{j=1\\ (j,n)_{**}=1}}^n \left(x- \zeta_n^j \right),
\end{equation*}
where $(j,n)_{**}$ stands for the greatest common unitary divisor of $j$ and $n$. The degree of the polynomial $\Phi^{**}_n(x)$
equals $\varphi^{**}(n)$, the bi-unitary Euler function, which is defined as $\varphi^{**}(n)=\# \{j: 1\le j \le n, (j,n)_{**}=1\},$ see the
paper \cite{Tot2009}. Although these definitions seem to be more natural than the previous ones, the properties of $\Phi^{**}_n(x)$
and $\varphi^{**}(n)$ are not similar to their unitary analogs. For
example, the function $\varphi^{**}(n)$ is not multiplicative and the
coefficients of the polynomials $\Phi^{**}_n(x)$ are in
general not integers
(we have, e.g., $\Phi^{**}_6(x)=x^3-\overline{\eta}x^2+\overline{\eta}x+ \eta$, where $\eta=(1+i\sqrt{3})/2$ and $\overline{\eta}=(1-i\sqrt{3})/2$).

\subsection{Inclusion-exclusion and Kronecker polynomials}

Let $\rho=\{r_1,r_2,\ldots,r_s\}$ be a set of increasing natural numbers satisfying
$r_i>1$ and $(r_i,r_j)=1$ for $i\ne j$, and put
$$n_0=\prod_i r_i,~n_i=\frac{n_0}{r_i},~n_{ij}=\frac{n_0}{r_ir_j}~[i\ne j],\ldots$$
For each such $\rho$ we define a function $Q_{\rho}$ by
\begin{equation}
\label{Qformula}
Q_{\rho}(x)=\frac{(x^{n_0}-1)\cdot \prod_{i<j}(x^{n_{ij}}-1)\cdots}
{\prod_i (x^{n_i}-1)\cdot \prod_{i<j<k}(x^{n_{ijk}}-1)\cdots}.
\end{equation}
It can be shown that $Q_{\rho}(x)$ is a polynomial
of degree $n_0\prod_{r_i\mid n_0}(1-1/r_i)$ having integer
coefficients.
This class of polynomials was introduced by Bachman \cite{Bac2010}, who
named them \emph{inclusion-exclusion polynomials}.

A \emph{Kronecker polynomial} $f\in \mathbb Z[x]$ is
a monic polynomial having all its roots inside or on the
unit circle. It was proved by Kronecker, cf.\,\cite{Dam}, that
such a polynomial is
a product of a monomial and cyclotomics and so we can write
\begin{equation}
\label{Kroneckerdecomposition}
f(x)=x^s\prod_{d}\Phi_d(x)^{e_d},
\end{equation}
with $s,e_d\ge 0$ and $e_d\ge 1$ for only finitely many $d$.

We will show how a unitary cyclotomic can be realized as an inclusion-exclusion
cyclotomic. As $Q_{\rho}(x)$ is monic and in $\mathbb Z[x],$ it follows from
\eqref{Qformula} that it is Kronecker.
Thus we have the following inclusions:
\begin{equation}
    \label{inclusions}
\{\text{unitary cyclotomics}\}\subset
\{\text{inclusion-exclusion polynomials}\}\subset \{\text{Kronecker polynomials}\}.
\end{equation}

The inclusion-exclusion polynomials that
are unitary can be precisely identified (for the proof see
Section \ref{sec:relation}).
\begin{theorem}
\label{setequal}
The set of unitary polynomials $\Phi^{*}_n(x)$ with $n\ge 2$ equals the set of
inclusion-exclusion polynomials $Q_{\rho}(x)$ with $\rho$ having prime
power entries, with no base prime repeated. More precisely there
is a one-to-one map between these sets that sends $n$ to
$\rho=\{p_1^{e_1},\ldots, p_s^{e_s}\},$ where $p_1^{e_1}\cdots p_s^{e_s}$
 with $p_1^{e_1}<\ldots <p_s^{e_s}$ is the prime factorization of $n$, resulting
 in
 $$\Phi^{*}_n(x)=Q_{\{p_1^{e_1},\ldots, p_s^{e_s}\}}(x).$$
\end{theorem}

This theorem shows that the first inclusion in \eqref{inclusions} is strict, e.g., $Q_{\{5,6\}}(x)$ is not a unitary cyclotomic.
By Theorem \ref{Th_pol} (or Theorem \ref{Thm:Bachman})
any Kronecker polynomial divisible
by $\Phi_d(x)^2$ for some $d\ge 1$ cannot be
an inclusion-exclusion polynomial, and so also the second inclusion is strict.
Even more, it is easy to see that
for both inclusions the
set theoretic differences are
infinite.

We would like to point out that in this paper, with the
exception of Theorem \ref{thm:main2}, the nomination ``theorem'' is not used to indicate
a deep result, but rather a key fact.

\section{Elementary properties of unitary cyclotomic polynomials}
\label{basicunitary}
The polynomials $\Phi^*_n(x)$ have integer coefficients. This follows by induction on $n$ by taking into account identity \eqref{prod_Phi_star},
similar to the case of classical cyclotomic polynomials. Indeed, various of our arguments in
this section closely mirror those for cyclotomic polynomials and can, in somewhat more detail than
we provided, be found in Thangadurai \cite{Thanga}.

By the definition \eqref{def_unit_cyclotomic} and \eqref{unit_Mobius},
\begin{equation*}
\log \Phi^*_n(x)= \sum_{\substack{j=1\\ (j,n)_*=1}}^n \log \left(x- \zeta_n^j\right)= \sum_{j=1}^n
\log \left(x- \zeta_n^j \right) \sum_{d\mid\mid (j,n)_*} \mu^*(d).
\end{equation*}
Note that $d\mid\mid (j,n)_*$  holds if and only if $d\mid j$ and $d\mid\mid n$. Hence
\begin{equation*}
\log \Phi^*_n(x)= \sum_{d\mid\mid n} \mu^*(d) \sum_{k=1}^{n/d}
\log \left(x- \zeta_{n/d}^k \right) = \sum_{d\mid\mid n} \mu^*(d) \log \left(x^{n/d}-1\right),
\end{equation*}
giving \eqref{Phi_star}, which
by unitary M\"{o}bius inversion is equivalent to \eqref{prod_Phi_star}.

The unitary divisors of prime powers $p^a$ ($a \ge 1$) are $1$ and $p^a$. We deduce by \eqref{prod_Phi_star}
that
\begin{equation}
\label{totaalflauw}
\Phi^*_{p^a}(x)=\frac{x^{p^a}-1}{x-1}=\prod_{j=1}^a \Phi_{p^j}(x).
\end{equation}

From formula \eqref{Phi_star} we immediately see that
the Taylor series of $\Phi_n^*(x)$ around $x=0$ has
integer coefficients, showing again that the coefficients of $\Phi^*_n(x)$ have to be integers.

Using  \eqref{unit_Mobius}, we see that, for
$n>1,$ we can rewrite \eqref{Phi_star} as
\begin{equation} \label{Phi_star-1}
\Phi^*_n(x)=\prod_{d\mid\mid n} \left(1-x^{d}\right)^{\mu^*(n/d)}.
\end{equation}

From \eqref{Phi_star-1} and \eqref{Phi_star} it follows that for $n>1$
\begin{equation} \label{selfreciprocal}
\Phi^*_{n}(x)=x^{\varphi^*(n)}\Phi^*_{n}(1/x),
\end{equation}
in other words, unitary cyclotomics are
\emph{self-reciprocal}.

For odd $n>1,$ we have
\begin{equation} \label{2n}
\Phi^*_{2n}(x)=\Phi^*_{n}(-x).
\end{equation}
In order to prove this we invoke
\eqref{Phi_star-1} and
group the even and odd unitary divisors together.
This leads to
\begin{eqnarray*}
\Phi^*_{2n}(x) &=& \prod_{2d\mid \mid 2n}(1-x^{2d})^{\mu^*(n/d)} \prod_{d\mid \mid n}(1-x^{d})^{\mu^*(2n/d)};\\
&=& \prod_{d\mid \mid n}(1-x^{2d})^{\mu^*(n/d)} \prod_{d\mid \mid n}(1-x^{d})^{-\mu^*(n/d)};\\
&=&  \prod_{d\mid \mid n}(1+x^{d})^{\mu^*(n/d)} =\Phi^*_n(-x).
\end{eqnarray*}

Let $k\ge 1$ be an integer and $p\nmid n$ a prime. The unitary divisors
of $p^kn$ come in two flavors:
those of the form $p^kd$ with $d\mid \mid n,$ and
those of the form $d\mid \mid n.$ On grouping these
together we obtain from \eqref{Phi_star-1} that
\begin{equation} \label{pkn}
\Phi^*_{p^{k}n}(x)=\frac{\Phi^*_{n}(x^{p^{k}})}
{\Phi^*_{n}(x)}.
\end{equation}

Also,
\begin{equation}
\label{powerreduce}
\Phi^*_{p^kn}(x)=\prod_{j=0}^{k-1}\Phi^*_{pn}(x^{p^{j}}).
\end{equation}

To see this we write each of the terms appearing in right hand side as a quotient of two unitaries given by \eqref{pkn}. We so obtain a
quotient of two unitaries, which equals the left hand side of \eqref{powerreduce} by \eqref{pkn} again.

Let $\Phi^*_n(x)=x^{\varphi^*(n)}+ b_1x^{\varphi^*(n)-1}+\cdots +b_{\varphi^*(n)}$. It follows, similar to the classical case, that
$b_1= - c^*_n(1)= -\mu^*(n)$ for every $n\in \N$.

\section{Unitary cyclotomic polynomials as products of cyclotomic polynomials}

Recall that $\kappa(n)=\prod_{p\mid n}p$ is the square-free kernel of $n.$
\begin{theorem} \label{Th_pol}  For any natural number $n$ we have
\begin{equation} \label{Phi_Phi_star}
\Phi^*_{n}(x)= \prod_{\substack{d\mid n\\ \kappa(d)=\kappa(n)}} \Phi_d(x).
\end{equation}
\end{theorem}

\begin{proof}
Combination of \eqref{Phi_star} with   \eqref{factor} yields
\begin{equation} \label{blah}
\Phi^*_n(x)=\prod_{d\mid\mid n} \left(x^{d}-1\right)^{\mu^*(n/d)}=
\prod_{d\mid\mid n}\Big(\prod_{\delta \mid d}\Phi_{\delta}(x)\Big)^{\mu^*(n/d)}.
\end{equation}

We thus find that $\Phi^*_n(x)=\prod_{\delta\mid n}\Phi_{\delta}(x)^{e_{\delta}},$ with
\begin{equation} \label{e_delta}
e_{\delta}=\sum_{k\delta \mid \mid n}\mu^*\left(\frac{n}{k\delta}\right).
\end{equation}

The exponents $e_{\delta}$ are integers that are to be determined.
Given a divisor $\delta$ of $n,$ we let $d$ be the smallest
multiple of $\delta$ that is a block divisor of $n.$ Note that
if $k\delta \mid \mid n,$ then there is an integer $m$ such that
$k\delta=md.$ The condition $k\delta \mid \mid n$ is in general not
equivalent with $k\mid \mid n/\delta,$ however the condition
$md \mid \mid n$ is equivalent with $m\mid \mid n/d.$
Using these observations and \eqref{unit_Mobius} we conclude that
\begin{equation}
\label{edeltavaluation}
e_{\delta}=\sum_{k\delta \mid \mid n}\mu^*\left(\frac{n}{k\delta}\right)=
\sum_{md \mid \mid n}\mu^*\left(\frac{n}{md}\right)=
\sum_{m \mid \mid n/d}\mu^*\left(\frac{n}{md}\right)=\epsilon\left(\frac{n}{d}\right).
\end{equation}
It follows that $e_{\delta}=0,$ except when $n$
is the smallest
multiple of $\delta$ that is a block divisor of $n$
(which occurs if and only if $\kappa(\delta)=\kappa(n)$), in which case
$e_{\delta}=1.$
\end{proof}

\begin{remark} {\rm
An alternative form of
\eqref{Phi_Phi_star} is
\begin{equation}
\label{cycloreduction}
\Phi^*_n(x)=\prod_{d\mid \frac{n}{\kappa(n)}}\Phi_{\kappa(n)}(x^d),
\end{equation}
which is obtained on noting that
$$
\Phi^*_{n}(x)= \prod_{\substack{d\mid n\\ \kappa(d)=\kappa(n)}} \Phi_d(x)=
\prod_{\substack{d \kappa(n) \mid n}} \Phi_{d\kappa(n)}(x)=\prod_{d\mid \frac{n}{\kappa(n)}}\Phi_{\kappa(n)}(x^d),
$$
where in the last step we used repeatedly that $\Phi_{pn}(x)=\Phi_{n}(x^p)$
if $p\mid n.$ }
\end{remark}

\begin{remark}
{\rm Theorem \ref{setequal} says that $\Phi^*_n(x)$ is an inclusion-exclusion polynomial associated to
the prime power factorization of $n$. A formula of Bachman giving
the factorization of an inclusion-exclusion polynomial
in cyclotomic polynomials (Theorem \ref{Thm:Bachman}), then
leads to an alternative proof of Theorem \ref{Th_pol} (Section
\ref{sec:relation}).}
\end{remark}

\begin{remark} {\rm The convolution defined by
\begin{equation*}
(f\diamond g)(n) = \sum_{\substack{d\mid n \\ \kappa(d)=\kappa(n)}} f(d)g(n/d) \qquad (n\in \N)
\end{equation*}
was mentioned by Subbarao \cite{Sub1972} and investigated by Thrimurthy \cite{Thr1977}.
It preserves the multiplicativity of functions, although it is noncommutative and nonassociative. However, as it is easy to check,
for any arithmetic functions $f,g$ and $h$,
\begin{equation} \label{prop_diamond}
(f\diamond g)\diamond h = f\diamond (g*h),
\end{equation}
where $*$ is the Dirichlet convolution. See also the review
MR0480305 (58 \# 478) of \cite{Thr1977}.
}
\end{remark}

Our next theorem generalizes Theorem \ref{Th_pol}. Indeed, Theorem \ref{Th_pol} follows from \eqref{g_star} on
making the choice
$g(n)=\log \Phi_n(x)$ (hence
$f(n)=\log(x^n-1)$) and $g^*(n)=\log \Phi^*_n(x)$.
In addition, with this choice \eqref{g}
yields the identity
\begin{equation*}
\Phi_n(x)= \prod_{\substack{d\mid n\\ \kappa(d)=\kappa(n)}} \Phi^*_d(x)^{\mu(n/d)} \qquad (n\in \N),
\end{equation*}
expressing a cyclotomic in terms of unitary cyclotomics. Note that if $d\mid n$, then $\kappa(d)=\kappa(n)$ holds iff $\kappa(n)\mid d$ iff
$\kappa(n/d)\mid d$.
\begin{theorem} \label{Th_general} Let $g,g^*:\N\to \C$ be arbitrary functions. Put $f(n)=\sum_{d\mid n} g(d).$
Assume that
\begin{equation*}
f(n)= \sum_{d\mid\mid n} g^*(d) \qquad (n\in \N).
\end{equation*}
Then
\begin{equation} \label{g_star}
g^*(n)= \sum_{\substack{d\mid n\\ \kappa(d)=\kappa(n)}} g(d) \qquad (n\in \N)
\end{equation}
and
\begin{equation} \label{g}
g(n)= \sum_{\substack{d\mid n\\ \kappa(d)=\kappa(n)}} g^*(d) \mu(n/d) \qquad (n\in \N).
\end{equation}
\end{theorem}

\begin{remark} {\rm In Theorem \ref{Th_general} the function $g$ is multiplicative if and only if $g^*$ is multiplicative.}

\end{remark}
\begin{proof}[Proof of Theorem \ref{Th_general}]
By M\"obius inversion we have
\begin{equation} \label{g_Mobius}
g(n)= \sum_{d\mid n} f(d)\mu(n/d)
\end{equation}
and
\begin{equation} \label{g_star_Mobius}
g^*(n)= \sum_{d\mid\mid n} f(d) \mu^*(n/d).
\end{equation}
These identities show that given $g,$ the function $g^*$ is uniquely
determined and reversely. We have
$$
g^*(n)= \sum_{d\mid \mid n} \mu^*(n/d) \sum_{\delta \mid d} g(\delta) =\sum_{\delta \mid n} g(\delta) e_{\delta},
$$
where $e_{\delta}$ is given by \eqref{e_delta}.
The proof of \eqref{g_star} is now easily completed
on invoking \eqref{edeltavaluation}, cf.\,the proof of Theorem \ref{Th_pol}.

Now we prove identity \eqref{g}. Put $f:=g$, $g:=\1$ (constant $1$ function), $h:=\mu$ in identity \eqref{prop_diamond}. This gives
\begin{equation*}
(g\diamond \1)\diamond \mu = g\diamond (\1*\mu).
\end{equation*}

Here, $g\diamond \1=g^*$ by \eqref{g_star}. Also, $\1*\mu=\epsilon,$ which is a basic property of the classical M\"obius function.
Since the function $\epsilon$ is the identity for the $\diamond$ operation, we
conclude that
\begin{equation*}
g^* \diamond \mu = g,
\end{equation*}
completing the proof.
\end{proof}

\section{Further properties of unitary cyclotomic polynomials}

\subsection{Calculation of $\Phi_n^*(\pm 1)$}
\label{flauweevaluatie}
In this section we determine $\Phi^*_n(\pm 1).$ For
completeness and comparison we mention the analogous classical results
for $\Phi_n(1).$

Let $\Lambda^*$ denote the unitary analog of the von Mangoldt function $\Lambda$. It is given by
\begin{equation}
\label{lambdastar}
\Lambda^*(n) = \begin{cases} a \log p & \text{ if $n=p^a$ is a prime power ($a\ge 1$);}\\ 0 & \text{ otherwise,}
\end{cases}
\end{equation}
and satisfies $\sum_{d \mid\mid n} \Lambda^*(d)=\log n$ (analogous to the classical identity
$\sum_{d \mid n} \Lambda(d)=\log n$).
\begin{lemma}
\label{valueat1A}
We have
$$\Phi_n(1)=\begin{cases}0 & \hbox{if } n=1;\\
p & \hbox{if } n=p^e;\\
1 & \text{otherwise},\end{cases}{\rm ~and~}\Phi^*_n(1)=\begin{cases}0 & \hbox{if } n=1;\\
p^e & \hbox{if } n=p^e;\\
1 & \text{otherwise},\end{cases}$$
with $p$ a prime number and $e\ge 1$.
\end{lemma}
In terms of the (unitary) von Mangoldt function this  can
be reformulated as follows.
\begin{lemma}
\label{valueat1B}
We have $\Phi_1(1)=0$ and $\Phi^*_1(1)=0$. For $n>1$ we have
$$\Phi_n(1)=e^{\Lambda(n)}{ ~and~}\Phi^*_n(1)=e^{\Lambda^*(n)}.$$
\end{lemma}
\begin{proof}[Proof of Lemma \ref{valueat1A}]
From \eqref{factor} and \eqref{prod_Phi_star} we obtain (respectively)
\begin{equation*}
\label{howtrivial!}
\frac{x^n-1}{x-1}=\prod_{d \mid n,~d>1}\Phi_d(x){\rm ~and~}\frac{x^n-1}{x-1}=\prod_{d \mid \mid n,~d>1}\Phi^*_d(x).
\end{equation*}
Thus (respectively)
\begin{equation}
\label{nisprod}
n=\prod_{d \mid n,~d>1}\Phi_d(1){\rm ~and~}n=\prod_{d\mid \mid n,~d>1}\Phi^*_d(1).
\end{equation}
By M\"obius inversion the
latter identities for all
$n>1$ determine $\Phi_m(1)$ and $\Phi^*_m(1)$ uniquely for all $m>1$.
It is thus enough to verify that the formulae claimed for $\Phi_m(1)$
and $\Phi^*_m(1)$ verify \eqref{nisprod}, which is evident.
\end{proof}
\begin{remark} {\rm
It is possible to prove Lemma \ref{valueat1B} with the (unitary) von
Mangoldt function naturally appearing in the proof. To do so one proceeds as in
the proof of Lemma  \ref{valueat1A} and deduces \eqref{nisprod} and
concludes that these equations uniquely determine $\Phi_m(1)$ and $\Phi^*_m(1)$.
It remains then (after taking logarithms) to prove the well-known (trivial) identity
$\log n=\sum_{d \mid n,~d>1}\Lambda(d)=\sum_{d \mid n}\Lambda(d),$ and
likewise in the unitary case,
$\log n=\sum_{d \mid \mid n,~d>1}\Lambda^*(d)=\sum_{d \mid \mid n}\Lambda^*(d).$}
\end{remark}

It is not much more difficult to evaluate $\Phi_n^*(-1).$

\begin{lemma}
\label{values_minus_1}
We have
$$\Phi^*_n(-1)=\begin{cases} -2 & \hbox{ if } n=1;\\ 0 & \hbox{ if } n=2^a; \\
p^b & \hbox{ if } n=2^a p^b;\\
1 & \text{ otherwise},\end{cases}$$
with $p$ an odd prime and $a,b\ge 1$.
\end{lemma}

\begin{proof} Follows from the identity
\eqref{cycloreduction}, Lemma \ref{valueat1A} and the well-known result
\begin{equation}
\label{negative}
\Phi_n(-1)=\begin{cases}-2 & \hbox{ if } n=1;\\ 0 & \hbox{ if } n=2;\\
p & \hbox{if } n=2p^e;\\
1 & \hbox{otherwise,}\end{cases}
\end{equation}
with $p\ge 2$ a prime number and $e\ge 1$.

Assume that $\Phi^*_n(-1)\ne 1.$ By \eqref{cycloreduction} it follows
that $\Phi^*_n(-1)=\Phi_{\kappa(n)}(-1)^e \Phi_{\kappa(n)}(1)^f,$
for some integers $e,f\ge 0.$ The formulas for $\Phi_n(\pm 1)$ then
show that $\kappa(n)|2p,$ with $p$ an odd prime, and so
$n=2^ap^b,$ $a,b\ge 0.$ In case $a,b\ge 1,$ we have
$\Phi^*_n(-1)=\Phi_{2p}(-1)^b \Phi_{2p}(1)^{ab-b}=
\Phi_{2p}(-1)^b=p^b,$
by, respectively, \eqref{cycloreduction}, Lemma \ref{valueat1A} and \eqref{negative}.
The remaining cases are left to the reader.
\end{proof}

Using identity \eqref{cycloreduction} one can likewise immediately
evaluate $\Phi^*_n(1)$ from $\Phi_n(1)$.

\subsubsection{Some products involving the cos and sin functions}
It is known that for any $n\ge 2$,
\begin{equation*} 
\prod_{\substack{j=1\\ (j,n)=1}}^n \sin \frac{\pi j}{n} = \frac{\Phi_n(1)}{2^{\varphi(n)}},
\end{equation*}
\begin{equation} \label{prod_cos}
\prod_{\substack{j=1\\ (j,n)=1}}^n \cos \frac{\pi j}{n} = \frac{\Phi_n(-1)}{(-4)^{\varphi(n)/2}},
\end{equation}
proved in \cite{DS1987} (for \eqref{prod_cos} in case $n$ is odd only) and \cite{Ten2005} (for any $n\ge 2$) by two different methods, see also \cite{LW2009}. Here we
provide the unitary analogs of these products, which 
in combination with the results of the previous section allow one to explicitly evaluate them.

\begin{thm} For any $n\ge 2$,
\begin{equation} \label{prod_sin_star}
\prod_{\substack{j=1\\ (j,n)_*=1}}^n \sin \frac{\pi j}{n} = \frac{\Phi^*_n(1)}{2^{\varphi^*(n)}},
\end{equation}
\begin{equation} \label{prod_cos_star}
\prod_{\substack{j=1\\ (j,n)_*=1}}^n \cos \frac{\pi j}{n} = \frac{\Phi^*_n(-1)}{(-4)^{\varphi^*(n)/2}}.
\end{equation}
\end{thm}

\begin{proof} We adapt the approach in \cite{Ten2005} to the unitary case. We need the simple formula
\begin{equation} \label{sum_j_star}
S^*(n):= \sum_{\substack{j=1\\ (j,n)_*=1}}^n j = \frac{n\varphi^*(n)}{2} \quad (n\ge 2),
\end{equation}
which can be shown similarly to the usual case.
We will only prove \eqref{prod_cos_star}, the proof of \eqref{prod_sin_star}
being similar. The product in the left hand side of \eqref{prod_cos_star} we
denote by $P^*(n).$

If $n=2^a,$ $a\ge 1,$ then we note that $(2^{a-1},2^a)_*=1$ and so the product
in \eqref{prod_cos_star} is zero. By Lemma \ref{values_minus_1} it follows
that also $\Phi_n^*(-1)$ is zero and thus in this case \eqref{prod_cos_star} holds.
Therefore we may assume that $n$ has an odd prime factor, which implies that $\varphi^*(n)$ is
even.
By \eqref{def_unit_cyclotomic} and \eqref{sum_j_star} we then see that
\begin{eqnarray*}
\Phi^*_n(-1) & =&\prod_{(j,n)_*=1} \left(-1 - \zeta_n^j \right)
= \prod_{(j,n)_*=1} \left(-\zeta_n^{j/2}
\right)\left(\zeta_n^{j/2}+\zeta_n^{-j/2}\right)\\
& =& 2^{\varphi^*(n)}P^*(n)\prod_{(j,n)_*=1} \left(-\zeta_n^{j/2}
\right)=(-2)^{\varphi^*(n)}\zeta_n^{S^*(n)/2}P^*(n)\\
& = & (-2)^{\varphi^*(n)} \zeta_n^{n\varphi^*(n)/4}P^*(n)
= (-2i)^{\varphi^*(n)}P^*(n),\\
& = & (-4)^{\varphi^*(n)/2}P^*(n),
\end{eqnarray*}
completing the proof of \eqref{prod_cos_star}.
\end{proof}

\begin{remark} {\rm
A completely similar argument leads to a proof of \eqref{prod_cos}.
The argument in that case is even easier, as $\varphi(n)$ is even for $n\ge 3$ and
it is not necessary to deal with the powers of two separately.}
\end{remark}

\subsubsection{Calculation of $\Phi_n^*$ at other roots of unity}
It is known how to explicitly evaluate $\Phi_n(\zeta_m)$
for $m\in \{3,4,5,6,8,10,12\},$ see \cite{BHMaa}. This in
combination with identity \eqref{cycloreduction} then allows one
to evaluate $\Phi_n^*(\zeta_m)$ for these values of
$m.$

\subsection{Unitary version of Schramm's identity}
In this section $x$ will
be a real variable. It was proved by Schramm \cite{Sch2015} that
\begin{equation*} 
\Phi_n(x)= \prod_{j=1}^n \left(x^{(j,n)}-1 \right)^{\cos(2\pi j/n)} \quad (x>1,\,n\in \N).
\end{equation*}

We will prove the following unitary analog.

\begin{theorem}
\label{form}
We have
\begin{equation*}
\Phi_n^*(x)= \prod_{j=1}^n \left(x^{(j,n)_*}-1 \right)^{\cos(2\pi j/n)} \quad (x>1,\,n\in \N).
\end{equation*}
\end{theorem}
This is, in fact, a corollary of a more general identity concerning the discrete Fourier transform (DFT)
\begin{align} \label{def_F*}
F^*_f(m,n):= \sum_{k=1}^n f((k,n)_*) \zeta_n^{km}
\end{align}
of functions involving the quantity $(k,n)_*$.

\begin{theorem} 
Let $f$ be an arbitrary arithmetic function. For every $m,n\ge 1$,
\begin{equation}
    \label{firstidentity}
F^*_f(m,n) = \sum_{d\mid (m,n)_*} d\, (\mu^* \times f)(n/d).
\end{equation}
Furthermore, we have
\begin{equation}
\label{secondidentity}
F^*_f(m,n)=\sum_{d\mid\mid n} f(d) c_{n/d}^*(m).
\end{equation}
\end{theorem}

\begin{proof} We have by using that $d\mid\mid (k,n)_*$ if and only if $d\mid k$ and $d\mid \mid n$,
\begin{equation*}
F^*_f(m,n)= \sum_{k=1}^n \zeta_n^{km} \sum_{d\mid\mid (k,n)_*} (\mu^* \times f)(d)
\end{equation*}
\begin{equation*}
= \sum_{d\mid\mid  n} (\mu^* \times f)(d) \sum_{j=1}^{n/d} \zeta_{n/d}^{jm}
= \sum_{\substack{d\mid\mid n\\ (n/d) \mid m}} (\mu^* \times f)(d) \frac{n}{d},
\end{equation*}
which proves \eqref{firstidentity}.

By grouping the terms according to the values of $(k,n)_*=d,$ we have
$$F^*_f(m,n)=\sum_{d\mid \mid n}f(d)\sum_{\substack{r=1\\ (r,n/d)_*=1}}^{r=n/d}\zeta_{n/d}^{rm}=\sum_{d\mid n}f(d)c_{n/d}^*(m),$$
which proves \eqref{secondidentity},
\end{proof}
\begin{remark}
{\rm By \eqref{unit_c_1} and \eqref{unit_c_2} the unitary Ramanujan sum
satisfies
$c_{\DOT}^*(m)= \varrho_{\DOT}(m) \times \mu^*$. Thus, the sum in \eqref{firstidentity} equals
$(\varrho_{\DOT}(m)\times \mu^* \times f)(n)= (c_{\DOT}^*(m) \times f)(n),$
leading to another proof of \eqref{secondidentity}.}
\end{remark}

\begin{remark} {\rm Identity \eqref{firstidentity} shows that if $f$ is a real valued function, then
so is $F^*_f(m,n)$ and hence in this case
the factor $\zeta_n^{km}$ in \eqref{def_F*}
can be replaced by $\cos(2\pi km/n)$. More exactly, if $f$ is a real valued function, then
\begin{align}
\label{cospart}
\sum_{k=1}^n f((k,n)_*) \cos(2\pi km/n) &= \sum_{d\mid (m,n)_*} d\, (\mu^* \times f)(n/d),\\
\sum_{k=1}^n f((k,n)_*) \sin(2\pi km/n) & = 0.\nonumber
\end{align}
In the special case $f(n)=n$ ($n\in \N$) and $m=1$ we obtain the following identities:
$$
\sum_{k=1}^n (k,n)_* \cos(2\pi k/n) =
n\sum_{d\mid \mid n}\frac{\mu^*(d)}{d}=\varphi^*(n)
\text{~and~}\sum_{k=1}^n (k,n)_* \sin(2\pi k/n)= 0.$$
In the classical case where $(k,n)_*$ is replaced by $(k,n)$ and $\varphi^*(n)$ is replaced
by $\varphi(n)$, these were
pointed out by Schramm \cite{Sch2008,Sch2015}.}
\end{remark}

\begin{proof}[Proof of Theorem \ref{form}]
By taking $f(n)=\log (x^n-1)$ we have by
\eqref{Phi_star} that
\begin{align*}
(\mu^* \times f)(n)= \sum_{d\mid \mid n} \mu^*(d) \log (x^{n/d}-1) =\log \Phi^*_n(x).
\end{align*}
The assumption that $x>1$ ensures that $f$ is real.
It then follows from \eqref{cospart} that
\begin{align*} 
\prod_{j=1}^n \left(x^{(j,n)_*}-1 \right)^{\cos(2\pi jm/n)} = \prod_{d\mid\mid (m,n)_*} \Phi^*_{n/d}(x)^d   \quad (m,n\in \N).
\end{align*}
The proof is completed
on putting $m=1$.
\end{proof}

\section{The coefficients of unitary cyclotomic polynomials}
\noindent We write
\begin{equation}
    \label{coeffdef}
\Phi_n^*(x)
=\sum_{j=0}^{\infty}a_n^*(j)x^j.
\end{equation}
This notation looks perhaps strange to the reader, but implicitly
defines the coefficients for every $j,$ which serves our purposes. In \cite{MMT} the following result is proven.
\begin{theorem}
\label{main}
Let $m\ge 1$ be fixed.
We have
$\{a^*_{mn}(j):n\ge 1,~j\ge 0\}=\mathbb Z.$
\end{theorem}
Given any polynomial $f,$ its \emph{height} $h(f)$
is defined as its
maximum coefficient in absolute value.
\begin{Con}
\label{cycloheight}
For any given natural number $m$ there is a
cyclotomic polynomial having height $m.$
\end{Con}
This conjecture was put forward by
Kosyak et al.\,\cite{KMSZ}. Here we propose the
following conjecture.
\begin{Con}
\label{unitarycycloheight}
For any given natural number $m$ there is a
unitary cyclotomic polynomial having height $m.$
\end{Con}
These two conjectures are closely
connected.
\begin{prop}
If Conjecture \ref{cycloheight} is true, then so
is Conjecture \ref{unitarycycloheight}.
\end{prop}
\begin{proof}
Suppose that $h(\Phi_n)=m.$ Then, by elementary
properties of cyclotomic polynomials,
$h(\Phi_{\kappa(n)})=m.$ Now note that
$h(\Phi^*_{\kappa(n)})=h(\Phi_{\kappa(n)})$
(since $\Phi^*_{\kappa(n)}=\Phi_{\kappa(n)}$).
\end{proof}
The best result available to date regarding these two conjectures is the following.
\begin{thm}
\label{thm:main2}
Almost all positive integers occur
as the height of an (unitary) cyclotomic polynomial. Specifically, for any fixed $\epsilon>0$,
the number of positive
integers $\le x$ that do not occur
as a height is $\ll_{\epsilon} x^{3/5+\epsilon}.$ Under the Lindel\"of Hypothesis this number
is $\ll_{\epsilon} x^{1/2+\epsilon}.$
\end{thm}
\begin{proof}
The result is actually a corollary of
\cite[Theorem 4]{KMSZ}. In that theorem only
certain special cyclotomic polynomials of the form $\Phi_{pqr},$
with $p<q<r$ primes, feature. As $\Phi^*_{pqr}=\Phi_{pqr},$
we are done.
\end{proof}
This theorem is deep, as it relies on deep results
from analytic number theory on gaps between
consecutive primes.

Let $k$ be a squarefree integer. Consider
the set
\begin{equation}
\label{bk}
\mathcal B(k):=\{h(\Phi^*_n):\kappa(n)=k\}.
\end{equation}
Note that if we replace $h(\Phi^*_n)$ by
$h(\Phi_n)$ this set will be $\{h(\Phi_k)\}.$
\begin{lemma}
Let $k$ be a squarefree integer.
Suppose that $k$ has at most two distinct prime factors.
Then ${\mathcal B}(k)=\{1\}.$
\end{lemma}
\begin{proof} If $k$ is a prime power, the conclusion follows from \eqref{totaalflauw}. 
If $k$ has precisely two distinct prime
factors, say $p$ and $q,$ we can write $\Phi^*_n(x)=Q_{\{p^e,q^f\}}(x)$ by Theorem
\ref{setequal}. The polynomial $Q_{\{p^e,q^f\}}(x)$ can be interpreted as
the semigroup polynomial associated to the numerical semigroup $\langle p^e,q^f\rangle$ and
as such will have height 1
(see Jones et al.\,\cite{MMT} or Moree \cite{Mor2014}). A shorter, but less conceptual, proof is obtained on
merely invoking Lemma 4 of \cite{DM}.
\end{proof}
\begin{remark} {\rm 
By the above proof and \eqref{2n} we also have
$h(\Phi^*_{2p^eq^f})=1.$ However, it is not always
true that $h(\Phi^*_{4p^eq^f})=1$ 
(for example $h(\Phi^*_{60})=2$).}
\end{remark}
Computer work by Bin Zhang suggests that
$\mathcal B(k)$ will be large if $k$ has at least three prime factors. This suggests that perhaps there is hope of
proving a stronger result on heights assumed by unitary cyclotomic polynomials
than is provided by Theorem \ref{thm:main2}.
\begin{question}
Suppose that $k$ has at least three odd prime factors.
Is $\mathcal B(k)$ unbounded?
\end{question}
If $k$ has four or more prime factors, we would not be surprised if
$\mathcal B(k)$ is unbounded. If it has precisely three factors, the
situation is not so clear. For example, 
$$\max\{h(\Phi^{*}_{n}): n= 2^a 3^b 5^c, ~a>0,~b>0,~c>0,~1<n<10^7\}=15.$$

\section{More on inclusion-exclusion polynomials}

The following result gives the
factorization of an inclusion-exclusion
polynomial $Q_{\rho}(x)$ into cyclotomics.
\begin{thm}[Bachman, \cite{Bac2010}]
\label{Thm:Bachman}
Let $\rho=\{r_1,r_2,\ldots,r_s\}$ be a set of increasing natural numbers satisfying
$r_i>1$ and $(r_i,r_j)=1$ for $i\ne j$.
Put $$D_{\rho}=\{d:d|\prod_{i=1}^{s} r_i\text{~and~}(d,r_i)>1\text{~for~all~}i\}.$$
Then $Q_{\rho}(x)=\prod_{d\in D_{\rho}}\Phi_d(x)$.
\end{thm}

Our proof of Theorem
\ref{setequal} makes use of the following basic
property of inclusion-exclusion polynomials.
\begin{thm}
\label{uniek}
Suppose that $Q_{\rho_1}(x)=Q_{\rho_2}(x),$ then
$\rho_1=\rho_2.$
\end{thm}

Its proof rests on the following easy lemma.
\begin{lemma}
\label{lem:Kroneckerexpansion}
A Kronecker polynomial can be written as
\begin{equation}
\label{product}
f(x)=x^s\prod_{d\in {\cal D}}(x^d-1)^{e_d},
\end{equation}
where  ${\cal D}$ is a unique finite set of integers and
$s\ge 0$ and the $e_d\ne 0$ are unique integers.
\end{lemma}
\begin{proof}[First proof]
From \eqref{Kroneckerdecomposition} and \eqref{Phi} we deduce that $f(x)$ can
be written as claimed. It remains to show the uniqueness. The integer $s\ge 0$
is merely the order of $f(x)$ in $x=0.$
We start by taking $\cal D$ to be the empty set.
We now consider
$f(x)/x^s$ around $x=0.$ Either $f(x)=1$ and we are done, or for some $\epsilon_1 \in \{-1,1\}$
we have
$\epsilon_1 f(x)=-1+ax^{d_1}+O(x^{d_1+1})$
with $a\ne 0$ an integer. In case $a$ is odd we have $(x^{d_1}-1)^{a}=-1+ax^{d_1}+O(x^{d_1+1}),$
in case $a$ is even we have $(x^{d_1}-1)^{a}=1-ax^{d_1}+O(x^{d_1+1}).$ We add $d_1$ to
$\cal D$ and put $e_{d_1}=a.$
Put $g(x):=f(x)(x^{d_1}-1)^{-a}.$
Note that either $g(x)=1$ or for an appropriate
$\epsilon_2 \in \{-1,1\}$ the identity  $\epsilon_2 g(x)=-1+bx^{d_2}+O(x^{d_2+1}),$
with $d_2>d_1$ and $b\ne 0$ holds. We now add $d_2$ to $\cal D$ and put
$e_{d_2}=b.$ We continue in this way until we arrive at the polynomial 1 (and this
will happen as we know a priori that $\cal D$ is finite).
\end{proof}
\begin{proof}[Second proof]
By \eqref{Kroneckerdecomposition}
we can write $f(x)=x^s(x-1)^{e_1}g(x),$
with $g(x)=\prod_{d\ge 2}\Phi_d(x)^{e_d}.$ As
$g(x)$ is the product of
cyclotomic $\Phi_d$ with $d\ge 2,$ it is selfreciprocal and satisfies $g(0)=1.$
Note that $s$ and $e_1$ are the order of $f(x)$ in
$x=0,$ respectively $x=1,$ and therefore unique.
Thus, w.l.o.g., we may assume that $f(x)$ is selfreciprocal and satisfies $f(0)=1.$
Around $x=0$
we either have
$f(x)=1$ or $f(x)=1-ax^{d_1}+O(x^{d_1+1})$
with $a\ne 0$ an integer and $d_1\ge 1.$
Note that $f(x)=(1-x^{d_1})^a+O(x^{d_1+1}).$
We start with setting $\cal D$ to be the empty set.
We add $d_1$ to
$\cal D$ and put $e_{d_1}=a.$
Put $g(x):=f(x)(1-x^{d_1})^{-a}.$
Note that either $g(x)=1$ or
$g(x)=1-bx^{d_2}+O(x^{d_2+1})$
with $d_2>d_1$ and $b\ne 0$. We now add $d_2$ to $\cal D$ and put
$e_{d_2}=b.$ We continue in this way until we arrive at the polynomial 1 (and this
will happen as we know a priori that $\cal D$ is finite).
We obtain
$$f(x)=\prod_{d\in {\cal D}}(1-x^d)^{e_d}
=\prod_{d\in {\cal D}}(x^d-1)^{e_d},$$
with ${\cal D}$ and the $e_d$ unique.
\end{proof}

\begin{remark} {\rm By a similar reasoning one can show that if
$f(x)\in\mathbb Z[\![x]\!]$ satisfies $f(x)\equiv 1\,({\rm mod~}x),$ there exist unique integers $e_1,e_2,\ldots$ such that
in $\mathbb Z[\![x]\!]$
$$
f(x)=\prod_{n=1}^{\infty}(1-x^n)^{e_n}.
$$
This is the so-called \emph{Witt expansion} which arises in many areas of mathematics, see
for example Moree \cite{Automata} for more information.}
\end{remark}

\begin{proof}[Proof of Theorem \ref{uniek}]
Given an inclusion-exclusion polynomial $f(x)$ in the standard form $f(x)=\sum_i a_ix^i,$ we
give a procedure leading to a unique $\rho=\{r_1,\ldots,r_m\}$ such that $f(x)=Q_{\rho}(x).$
As $f(x)$ is a Kronecker polynomial satisfying $f(0)=1,$
by Lemma \ref{lem:Kroneckerexpansion} we can
write
$$f(x)=\prod_{d\in {\cal D}}(x^d-1)^{e_d},$$
where ${\cal D}$ and the $e_d\ne 0$ are unique and computable.
Comparison with \eqref{Qformula} then shows
that ${\cal D}=\{n_0,n_i,n_{ij},n_{ijk},\ldots \}.$
The question is know which $\rho$
correspond to the set ${\cal D}$. We know for example,
that whatever
$\rho$ is, the product of its $r_i$ is certainly unique
and equals $n_0=\max {\cal D}.$
The numbers $r_i$ themselves, also turn out to be unique. We order the
elements in ${\cal D}$ in such a way that $d_1 < d_2 < \ldots.$ We put
$r_1=d_1.$ If $d_2$ is coprime with $r_1$ we put $r_2=d_2,$ if not we consider
$d_3.$ If an $d_i$ is coprime to every $r_1,\ldots,r_g$ we have at a certain point, we
put $r_{g+1}=d_i.$ If $r_m$ is the last $r$ number so found, we have
$f(x)=Q_{\rho}(x)$ with $\rho=\{r_1,\ldots,r_m\}$ uniquely determined.
\end{proof}

\subsection{Unitary cyclotomics as inclusion-exclusion polynomials}
\label{sec:relation}
Armed with Theorem \ref{uniek} we are now ready to prove Theorem \ref{setequal}.

\begin{proof}[Proof of Theorem \ref{setequal}]
Let $n\ge 2$ be an integer and $\prod_{i=1}^s p_i^{e_i}$ its canonical
factorization with $p_1^{e_1}<p_2^{e_2}
<\ldots <p_s^{e_s}$. On comparing formula
\eqref{Phi_star} with \eqref{Qformula} it
follows that $\Phi_n^*(x)=Q_{\{p_1^{e_1},\ldots, p_s^{e_s}\}}(x)$.
Reversely, given any ascending sequence of prime powers $p_1^{e_1}< \ldots <p_s^{e_s}$
with distinct base primes $p_1,\ldots,p_s,$ the polynomial
$Q_{\{p_1^{e_1},\ldots,\, p_s^{e_s}\}}(x)$ is seen to correspond to
$\Phi_n^*(x)$ with $n=\prod_{i=1}^s p_i^{e_i}.$ The one-to-one
part of the claim is a consequence of Theorem \ref{uniek}.
\end{proof}

Theorem \ref{setequal} together with Theorem  \ref{Thm:Bachman} can
be used to reprove Theorem
\ref{Th_pol}:
\begin{proof}[Alternative proof of Theorem {\rm \ref{Th_pol}}]
For $n=1$ the result is obviously true.
Now let $n\ge 2$ be an integer and $\prod_{i=1}^s p_i^{e_i}$ its canonical
factorization with $p_1^{e_1}<p_2^{e_2}
<\ldots <p_s^{e_s}$. Put $\rho=\{p_1^{e_1},\ldots,p_s^{e_s}\}$. Then
by Theorem \ref{setequal}, respectively Theorem  \ref{Thm:Bachman},
$$\Phi_n^*(x)=Q_{\rho}(x)=\prod_{d\in D_{\rho}}\Phi_d(x),$$ with
$D_{\rho}=\{d:d\mid n,~p_1\cdots p_s\mid d\}=\{d:d\mid n,~k(d)=k(n)\}$.
\end{proof}

\section{Applications of Theorem  \ref{Th_general}}

By selecting $g(n)=c_n(k)$, $g^*(n)=c^*_n(k)$ (Ramanujan, resp. unitary Ramanujan sums), where
$f(n)=\varrho_n(k)$, we deduce from Theorem \ref{Th_general} the following identities.

\begin{cor} \label{Cor_Ram} For any $n,k\in \N$ we have
\begin{equation} \label{c_star_c}
c^*_n(k)= \sum_{\substack{d\mid n\\ \kappa(d)=\kappa(n)}} c_d(k)
\end{equation}
and
\begin{equation*}
c_n(k)= \sum_{\substack{d\mid n\\ \kappa(d)=\kappa(n)}} c^*_d(k) \mu(n/d).
\end{equation*}
\end{cor}

Let $\id_s(n)=n^s$ ($s\in \R$). In the case $g(n)=J_s(n):=(\mu*\id_s)(n)$ (Jordan function of order $s$),
$g^*(n)=J^*_s(n):=(\mu^* \times \id_s)(n)$, the unitary Jordan function of order $s$, where
$f(n)=n^s$, we deduce the
following corollary of Theorem \ref{Th_general}.
\begin{cor} \label{Cor_Jordan} For every $n\in \N$, $s\in \R$ we have
\begin{equation} \label{J_star_J}
J^*_s(n)= \sum_{\substack{d\mid n\\ \kappa(d)=\kappa(n)}} J_s(d)
\end{equation}
and
\begin{equation*}
J_s(n)= \sum_{\substack{d\mid n\\ \kappa(d)=\kappa(n)}} J^*_s(d) \mu(n/d).
\end{equation*}
\end{cor}

Identity \eqref{c_star_c} was deduced by McCarthy \cite[Ch.\ 4]{McC1986}, while
\eqref{J_star_J} was obtained by Cohen \cite[Lemma 3.1]{Coh1961} using different reasonings.
If $s=1$, then $J_1(n)=\varphi(n)$, $J^*_1(n)=\varphi^*(n)$ and we deduce the
following corollary.
\begin{cor} We have
\begin{equation*}
\varphi^*(n)= \sum_{\substack{d\mid n\\ \kappa(d)=\kappa(n)}} \varphi(d) \qquad (n\in \N).
\end{equation*}
\end{cor}
This result also follows by setting $k=n$ in Corollary \ref{Cor_Ram}, or by comparing the degrees of the polynomials in \eqref{Phi_Phi_star}.

We remark that Cohen \cite[Lemma 4.1]{Coh1961} also showed that the identity
\begin{equation*}
\kappa^s(n)= \sum_{\substack{d\mid n\\ \kappa(d)=\kappa(n)}} d^s \mu^2(d) \qquad (n\in \N)
\end{equation*}
holds for any $s$.

Putting $k=1$ in Corollary \ref{Cor_Ram} and
noting that $\kappa(n)$ is the only squarefree
divisor $d$ of $n$ satisfying $\kappa(d)=\kappa(n),$
we obtain the following corollary.
\begin{cor} We have
\begin{equation*}
\mu^*(n)= \sum_{\substack{d\mid n\\ \kappa(d)=\kappa(n)}} \mu(d)=\mu(\kappa(n)) \qquad (n\in \N).
\end{equation*}
\end{cor}

On taking $g(n)=\Lambda(n)$, $g^*(n)=\Lambda^*(n)$
and by the well-known identity $\sum_{d \mid n}\Lambda(d)=\log n,$
we see that $f(n)=\log n$ and obtain the final corollary.
\begin{cor} We have
\begin{equation*}
\Lambda^*(n)= \sum_{\substack{d\mid n\\ \kappa(d)=\kappa(n)}} \Lambda(d) \qquad (n\in \N).
\end{equation*}
\end{cor}
From it we deduce
the truth of \eqref{lambdastar}.

\section{Connection with unitary Ramanujan sums}

Certain formulas concerning the unitary cyclotomic polynomials and unitary Ramanujan sums can be easily deduced from their classical
analogues, by using the above identities. For example, we have

\begin{cor} For any $n>1$ and $x\in \C$, $|x|<1$,
\begin{equation}
\Phi^*_n(x) = \exp \left(-\sum_{k=1}^{\infty} \frac{c^*_n(k)}{k}x^k \right).
\end{equation}
\end{cor}

\begin{proof} It is known that for any $n>1$ and $|x|<1$,
\begin{equation} \label{cyclotomic_Ramanujan_exp}
\Phi_n(x) = \exp \left(-\sum_{k=1}^{\infty} \frac{c_n(k)}{k}x^k \right),
\end{equation}
see Nicol \cite[Cor.\ 3.2]{Nic1962}, T\'oth \cite[Th.\ 1]{Tot2010} or Herrera-Poyatos
and Moree \cite{AndresPieter}. By Theorem  \ref{Th_pol} and the identity
\eqref{c_star_c} we then obtain
 $$\Phi^*_{n}(x)= \prod_{\substack{d\mid n\\ \kappa(d)=\kappa(n)}} \Phi_d(x)=\exp \left(-\sum_{k=1}^{\infty}\frac{x^k}{k}\sum_{\substack{d\mid n\\ \kappa(d)=\kappa(n)}} c_d(k)\right)=\exp \left(-\sum_{k=1}^{\infty} \frac{c^*_n(k)}{k}x^k \right),$$
completing the proof.
\end{proof}

For any $n>1$ the series $\sum_{k=1}^{\infty} c_n(k)/k$
converges, see, e.g., H\"older \cite{Hol1936}. Therefore, Lemma \ref{valueat1B} in combination with
\eqref{cyclotomic_Ramanujan_exp}
gives
\begin{equation} \label{Ramanaujan_Lambda}
\log \Phi_n(1)=\Lambda(n)= - \sum_{k=1}^{\infty} \frac{c_n(k)}{k}
\quad (n>1).
\end{equation}
The second identity was first discovered by Ramanujan and expresses an arithmetic function as an infinite series involving Ramanujan
sums (for a different proof see Sivaramakrishnan \cite[Theorem 87]{Siv1989}).
Such expressions are now called
\textit{Ramanujan expansions}, see, e.g.,
Schwarz and Spilker \cite[Chapter VIII]{ScSp}.

The convergence of $\sum_{k=1}^{\infty} c_n(k)/k$
for $n>1$ in combination with
\eqref{c_star_c} shows that also
$\sum_{k=1}^{\infty} c_n^*(k)/k$ converges for
any $n>1$.
Thus, by Lemma \ref{valueat1B} again,
\begin{equation} \label{unit_Ramanujan_Lambda}
\log \Phi^*_n(1) = \Lambda^*(n)= - \sum_{k=1}^{\infty} \frac{c^*_n(k)}{k}  \quad (n>1).
\end{equation}

The second identity in \eqref{unit_Ramanujan_Lambda} was obtained by Subbarao \cite{Sub1966}, without referring to unitary
cyclotomic polynomials and with an incomplete proof, namely without showing that the corresponding series converges.

\vskip2mm

\noindent {\tt Acknowledgments}. 
Part of this work was done when the second author visited the Max Planck Institute of
Mathematics in Bonn, Germany, in February 2017. He thanks the Institute for the invitation and support.
We thank Gennady Bachman for correspondence that led
to a simplification of the proof of Theorem \ref{uniek}. Bin Zhang (Qufu Normal University), kindly did computer
work on $\mathcal B(k)$ (defined in \eqref{bk}) on request
of the first author.

\vskip4mm

\noindent
Pieter Moree \\
Max-Planck-Institut f\"ur Mathematik \\
Vivatsgasse 7, 53111 Bonn, Germany  \\
E-mail: {\tt moree@mpim-bonn.mpg.de}
\vskip2mm

\noindent
L\'aszl\'o T\'oth  \\
Department of Mathematics \\
University of P\'ecs \\
Ifj\'us\'ag \'utja 6, 7624 P\'ecs, Hungary \\
E-mail: {\tt ltoth@gamma.ttk.pte.hu}

\end{document}